\documentclass[12pt]{amsart}

\usepackage{amssymb,amsmath,amsthm,hyperref,mathrsfs}
\usepackage[left=2.8cm,top=2.8cm,right=2.8cm,bottom=2.8cm]{geometry}
\usepackage{amsfonts}
\usepackage{paralist}
\usepackage{tikz}
\usepackage{subfigure}
\usepackage{times}
\usepackage{array}
\usepackage{multirow}
\usepackage{graphicx}

\usepackage{ifthen}
\newboolean{IncludeFragments}
\setboolean{IncludeFragments}{false}

\newif\ifnotesw\noteswtrue

\newtheorem{theorem}{Theorem}

\newtheorem{lemma}[theorem]{Lemma}

\newtheorem{corollary}[theorem]{Corollary}

\newcommand{\remove}[1]{}

\newcommand{\copwin}{\mathscr{C}}

\newcommand{\copmove}{{\rightarrow}}
\newcommand{\robmove}{{\rightarrow}}
\newcommand{\crmove}{{\twoheadrightarrow}}

\newcommand{\st}[2]{({#1}\,;\,{#2})}
\newcommand{\copst}[2]{(\underline{#1}\,;\,{#2})}
\newcommand{\robst}[2]{({#1}\,;\,\underline{#2})}

\newcommand{\bc}{C}

\begin{document}
%\addtolength{\baselineskip}{\baselineskip}
%\addtolength{\parskip}{\parskip}
%\spacing{1.4}
%\parskip=+3pt

\def\proofend{\hfill$\Box$\medskip}
\def\Proof{\noindent{\bf Proof. }}
\newcommand{\ProofOf}[1]{\noindent{\bf Proof of {#1}. }}

\newenvironment{proofof}[1]{\ProofOf{#1}}{\hfill $\Box$ \medskip}

\def\cN{\overline{N}}
\def\P{{\cal{P}}}
\def\cop{{\mathrm{cop}}}
\def\reg{{\mathrm{reg}}}

\title{On the minimum order of $k$-cop-win graphs}

\author[W. Baird]{William Baird}
\address{Department of Mathematics, Ryerson University, Toronto, ON, Canada, M5B 2K3}
\email{\tt liam.baird@gmail.com}
\author[A. Beveridge]{Andrew Beveridge}
\address{Department of Mathematics, Statistics and Computer Science, Macalester College, Saint Paul, MN, U.S.A.\ 55105}
\email{\tt abeverid@macalester.edu}
\author[A. Bonato]{Anthony Bonato}
\address{Department of Mathematics, Ryerson University, Toronto, ON, Canada, M5B 2K3}
\email{\tt abonato@ryerson.ca}
\author[P. Codenotti]{Paolo Codenotti}
\address{Institute for Mathematics and its Applications, University of Minnesota, Minneapolis, MN, U.S.A.\ 55455}
\email{\tt paolo@ima.umn.edu}
\author[A. Maurer]{Aaron Maurer}
\address{Department of Mathematics, Carleton College, Northfield, MN, U.S.A.\ 55057}
\email{\tt maurera@carleton.edu}
\author[J. McCauley]{John McCauley}
\address{Department of Mathematics, Haverford College, Haverford, PA, U.S.A.\ 19041}
\email{\tt ???}
\author[S. Valeva]{Silviya Valeva}
\address{Department of Mathematics,  University of Iowa, Iowa City, IA, U.S.A.\ 52242}
\email{\tt silviya-valeva@uiowa.edu}

\begin{abstract}
We consider the minimum order graphs with a given cop number. We prove
that the minimum order of a connected graph with cop number $3$ is $10$, and show that the Petersen graph is the unique isomorphism type of graph with this property. We provide the results of a computational search on the cop number of all graphs up to
and including order $10.$ A relationship is presented between the minimum order of graph with cop number $k$ and Meyniel's conjecture
on the asymptotic maximum value of the cop number of a connected graph.
\end{abstract}

\maketitle

\section{Introduction}

\emph{Cops and Robbers} is vertex-pursuit game played on graphs that has
been the focus of much recent attention. Throughout, we only consider finite, connected, undirected graphs.
There are two players consisting of
a set of \emph{cops} and a single \emph{robber}. The game is played over a
sequence of discrete time-steps or \emph{rounds}, with the cops going first
in the first round and then playing alternate time-steps. The cops and
robber occupy vertices, and more than one cop may occupy a vertex.
 When a player is ready to move in a round they must
move to a neighboring vertex. We include loops on each vertex so that
players can \emph{pass}, or remain on their own vertex. Observe that any
subset of cops may move in a given round. The cops win if after some finite
number of rounds, one of them can occupy the same vertex as the robber. This
is called a \emph{capture}. The robber wins if he can avoid capture
indefinitely. A \emph{winning strategy for the cops} is a set of rules that
if followed, result in a win for the cops, and a \emph{winning strategy for the
robber} is defined analogously.

If we place a cop at each vertex, then the cops are guaranteed to win.
Therefore, the minimum number of cops required to win in a graph $G$ is a
well defined positive integer, named the \emph{cop number} of the graph $G.$
We write $c(G)$ for the cop number of a graph $G$, and say that a graph
satisfying $c(G)=k$ is $k$-\emph{cop-win}. For example, the Petersen graph is $3$-cop-win. If $k=1,$ then we say that $G$ is \emph{cop-win}.
Nowakowski and Winkler~%
\cite{nw}, and independently Quilliot~\cite{q}, considered the game with one
cop only; the introduction of the cop number came in~\cite{af}. Many papers
have now been written on cop number since these three early works; see the
book \cite{bonato} for additional references and background on the cop
number.

Meyniel's conjecture is one of the deepest unsolved problems on the cop
number. It states that for a connected graph $G$ of order $n,$ $c(G)=O(\sqrt{%
n}).$ Hence, the largest cop number of a graph is asymptotically bounded above by $d\sqrt{n}$
for a constant $d.$ The conjecture has so far resisted all attempts to
resolve it, and the best known bounds (see, for example, \cite{lu}) do not even prove that $%
c(G)=O(n^{\varepsilon }),$ for $\varepsilon <1.$

The goal of the current study is to investigate the minimum order of graphs with a given cop number. For a fixed positive integer $k$, define $m_{k}$ to be the minimum order of a connected graph $G$ satisfying $c(G)\geq k.$ Define $M_{k}$ to be the minimum order of a connected $k$%
-cop-win graph. It is evident that the $m_{k}$ are monotonically increasing, and $m_{k}\leq M_{k}.$

Up until this study, only the first two values of these parameters where known: $m_1=M_1=1$ and $m_2=M_2=4$ (witnessed by the $4$-cycle). We derive that $m_3=M_3=10$; interestingly, the Petersen graph is the unique isomorphism type of $3$-cop-win graph with order $10.$ In addition to a proof of this fact, we performed a computer search to calculate the cop number of every connected graph on 10 or fewer vertices (there are nearly 12 million such unlabelled graphs). We performed this categorization by checking for cop-win orderings~\cite{nw} and using an algorithm provided in~\cite{bcp}. We present these computational results in the next section.

We prove the following theorems.

\begin{theorem}\label{thm:9-vtx-c2}
If $G$ is a graph on at most $9$ vertices, then $c(G)\leq 2$.
\end{theorem}

\begin{theorem}
 \label{thm:petersen}
The Petersen graph is the unique isomorphism type of graph on $10$ vertices that is $3$-cop-win.
\end{theorem}

The proofs of Theorems~\ref{thm:9-vtx-c2} and \ref{thm:petersen}---which are deferred to Section~\ref{sec:petersen}---exploit new ideas which are of interest in their own right. In particular, we prove a suite of structural lemmas concerning the cop number of graphs containing a vertex whose co-degree is a small constant, namely with maximum degree at least $n-7$, where $n$ is the order of the graph.

Further, we prove that Meyniel's conjecture is equivalent to bounds on the values $m_k;$ see Theorem~\ref{iii}. We give lower bounds on the growth rates of the number of non-isomorphic $k$-cop-win graphs of a given order in Theorem~\ref{fkbounds}.

For background on graph theory see \cite{west}. We use the notation $v(G)=|V(G)|.$ We use the notation $V$ and $E$ for the vertices and edges of a graph $G$, respectively, if $G$ is clear from context. For $u,v \in V$, we write $u\sim
v$ when $u v\in E$.  For $S \subseteq V$, we write $u \sim
S$ when $u \notin S$ and there exists $v \in S$ such that $u \sim
v$. Given a vertex $v$, its  \emph{neighborhood} is $N(v) = \{ u \in V \mid (v,u) \in E \}$, and its \emph{closed neighborhood} is $N[v] = \{ v \} \cup N(v)$. We define $N(S)= \bigcup_{v \in S} N(v) \setminus S$ and $N[S] = \bigcup_{v \in S} N[v]$.  For convenience, we use the notation $N(u,v) = N(\{u,v\})$. A vertex $v$ is \emph{dominated} by the vertex $w$ if $N[v] \subseteq N[w]$. For $S\subseteq V(G)$, the subgraph induced by $S$ is denoted $G[S]$. We use the notation $X\setminus Y$ for the difference of sets. If $H$ is an induced subgraph of $G$, then we write $G-H$ for $G[V(G)\setminus V(H)]$. For a set $S$ of vertices of $G$, we also write $G-S$ for $G[V(G)\setminus S].$
For sets of vertices $S, T\subseteq V$, we denote the set of edges
between the two sets by $[S : T] = \{s t\in E\mid s\in S, t\in
T\}$, and we use $|S:T|$ to denote the cardinality of this set. We denote the minimum degree by
$\delta(G)$ and the maximum degree by $\Delta(G)$.
We generalize the latter symbol to subsets of vertices: for $S\subseteq V$,
$\Delta(S) =\max_{s\in S} \deg(s)$.

%%%%%%%%%%%%%%%%%%%%%%%%%%%%%%%%%%%%%%%%%%%%%%%%%

\section{Computational results}
\label{comp}

We present the results of a computer search on the cop number of small order graphs. For a positive integer $n$, define $f_{k}(n)$ to be the number of non-isomorphic connected $k$-cop-win
graphs of order $n$ (that is, the \emph{unlabelled} graphs $G$ of order $n$ with $c(G)=k$). Define $g(n)$ to be the number of non-isomorphic (not
necessarily connected) graphs of order $n,$ and $g_c(n)$ the number of non-isomorphic connected graphs of order $n.$ Trivially, for all $k,$ $f_k(n)\le g(n).$ The following table presents the values of $g,$ $g_c,$ $f_1,$ and $f_2$ for small orders.

\begin{table}[!h]
\begin{center}
\begin{tabular}{|c||c|c|c|c|c|c|}
\hline
order $n$ & $g(n)$ & $g_{c}(n)$ & $f_{1}(n)$ & $f_{2}(n)$ & $f_3(n)$\\ \hline\hline
1 & 1 & 1 & 1 & 0 & 0 \\ \hline
2 & 2 & 1 & 1 & 0 & 0\\ \hline
3 & 4 & 2 & 2 & 0 & 0\\ \hline
4 & 11 & 6 & 5 & 1 & 0\\ \hline
5 & 34 & 21 & 16 & 5 & 0\\ \hline
6 & 156 & 112 & 68 & 44 & 0\\ \hline
7 & 1,044 & 853 & 403 & 450 & 0\\ \hline
8 & 12,346 & 11,117 & 3,791 & 7,326 & 0\\ \hline
9 & 274,668 & 261,080	&65,561	&195,519 & 0\\ \hline
10&  12,005,168 &  11,716,571 &  2,258,313  & 9,458,257 & 1 \\ \hline
\end{tabular}
\end{center}
\end{table}

\vspace{.1in}

The values of $g$ and $g_c$ come from \cite{sloane}, $f_1$ was computed by checking for cop-win orderings~\cite{nw}, while $f_2$ and $f_3$ were computed using Algorithm~1 in \cite{bcp}.
Among these graphs there is only one graph $G$ of order $10$ that requires 3 cops to win (which independently verifies Theorems~\ref{thm:9-vtx-c2} and \ref{thm:petersen}).  The graph $G$ must be the Petersen graph, since we know that it is 3-cop win.

\section{Meyniel's conjecture and growth rates}\label{meys}

The following theorem, while straightforward to prove, sets up an unexpected connection
between Meyniel's conjecture and the order of $m_k.$
\begin{theorem}\label{iii}
\begin{enumerate}
\item $m_k = O(k^2).$
\item Meyniel's conjecture is equivalent to the property that%
\[
m_{k}=\Omega (k^{2}).
\]
\end{enumerate}
\end{theorem}
\noindent Hence, if Meyniel's conjecture holds, then Theorem~\ref{iii} implies that $m_k = \Theta(k^2).$

\begin{proof}
The incidence graphs of projective planes have order $2(q^{2}+q+1),$ where $q$
is a prime power, and have cop number $q+1$; see \cite{bb} or \cite{pralat}.  Hence, this family of graphs show that for $q
$ a prime power,%
\[
m_{q+1}=O(q^{2})
\]
Now fix $k$ a positive integer. Bertrand's postulate (which states that all integers $x>1,$ there is a prime $q$ between $x$ and $2x$; see \cite{che,erdosb}) provides a prime $q$
with $k< q< 2k$. Hence,

\begin{equation*}
m_{k}\leq m_{q}\leq m_{q+1}=O(q^{2})=O((2k)^{2})=O(k^{2}).
\end{equation*}

For item (2), if $m_{k}=o(k^{2})$, then there is some connected graph $G$ with order $%
o(k^{2})$ and cop number $k$. But Meyniel's conjecture implies that $%
c(G)=o(k)$, a contradiction. Hence, Meyniel's conjecture implies that $%
m_{k}=\Omega (k^{2}).$

For the reverse direction, suppose that $m_{k}=\Omega (k^{2}).$ For a contradiction,
suppose that Meyniel's conjecture is false. Then there is a connected graph $%
G$ of order $n$ with $c(G)=k=\omega (\sqrt{n}).$ Then $\sqrt{n}=o(k)$, and
so $n=o(k^{2})$. But then $m_{k}\leq n=$ $o(k^{2}),$ a contradiction. \end{proof}

While the parameters $m_k$ are non-decreasing, an open problem is to determine whether the $M_k$ are in fact non-decreasing.
A possibly more difficult problem is to settle whether $m_k=M_k$ for all $k\ge 1.$
The gap in our knowledge of the parameters $m_{k}$ and $M_{k}$ points to the
question of growth rates for the classes of connected $k$-cop-win graphs.
It is well known (see \cite{sloane} for example) that
\begin{eqnarray*}
g(n) &=&(1+o(1))\frac{2^{\binom{n}{2}}}{n!} \\
&=&2^{\frac{1}{2}n^{2}-\frac{1}{2}n-n\log_2 n + n\log_2 e - \Theta(\log n)},
\end{eqnarray*}
where the second equality follows by Stirling's formula. The following theorem supplies a super-exponential lower bound to the parameters $f_k.$ A vertex is \emph{universal} if it is adjacent to all others. If $f:X\rightarrow Y$ is a function and $S\subseteq X,$ then we denote the restriction of $f$ to $S$ by $f\upharpoonright S.$

\begin{theorem}
\label{fkbounds}
\begin{enumerate}
\item For all $n>1,$ $g(n-1)\leq f_{1}(n).$

\item For $k>1,$ and all $n>2m_{k},$ $g(n-m_{k}-1)\leq f_{k}(n).$
\end{enumerate}
\end{theorem}

\begin{proof}
For item (1), fix a graph $G$ of order $n-1.$ Form $G^{\prime }$ by adding a
universal vertex to $G.$ If $G\ncong H,$ then it is an exercise to show that $G^{\prime }\ncong
H^{\prime }.$ Item (1) now
follows since $G^{\prime }$ is cop-win.

For item (2), given $G$ of order $n-m_{k}-1,$ form a graph $G^{+k}$ as
follows. First form $G^{\prime }$ with the new universal vertex labelled $x_G.$ Fix a $k$-cop-win graph $H$ of order $m_k$
(which we label as $H_{G}),$ and specify a fixed vertex $y_{G}$ of $H_{G}.$
Add the bridge $x_{G}y_{G}$ connecting $H_{G}$ to $G^{\prime }.$

We first claim that $G^{+k}$ is $k$-cop-win. We have that $c(G^{+k})\geq k,$
since a winning strategy for the robber if there are fewer than $k$ cops is
to remain in $H_{G}.$ To show that $c(G^{+k})\leq k,$ a set of $k$ cops
plays as follows. At the beginning of the game, one cop is on $x_G,$ while the
remaining cops stay in $G.$ Then $R$ cannot move to $G^{\prime }$ without
being caught, so the robber moves in $H_{G}.$ All the cops then move to $%
H_{G}$ and play their winning strategy there, with the following caveat. If $%
R$ moves outside $H_{G},$ then the cops play as if $R$ is on $y_G.$
Eventually, the robber is caught in $H_{G},$ or the robber is in $G^{\prime }
$ and at least one cop occupies $y_G.$ But then that cop moves to $x_G$ to win.

To finish the proof of (2), we must show that if $G\ncong J,$ then $%
G^{+k}\ncong J^{+k}.$ For a contradiction, let $h:G^{+k}\rightarrow J^{+k}$
be an isomorphism. Then we must have $h(x_{G})=x_{J}$ by noting that
$x_G$ and $x_J$ are the only vertices with the maximum degree $n-m_k$ (note that $y_G$ has degree at most $m_k < n -m_k$ by hypothesis). The vertex $y_{G}$ is unique with the property that it is adjacent to $x_{G}$ and has neighbors not adjacent to $x_{G}$ (the same holds by
replacing the subscript $G$ by $J).$ But then $h(H_{G})=H_{J},$ which
implies the contradiction that restricted mapping $h\upharpoonright G: G\rightarrow J$ is an
isomorphism.
\end{proof}

We do not know the asymptotic order for $f_k$ (even if $k=1$). A
recent result~\cite{bkp} proves that the number of distinct \emph{labelled} cop-win graphs is $2^{\frac{1}{2}n^2-\frac{1}{2}n+o(n)}.$

\section{Proofs of Theorems~\ref{thm:9-vtx-c2} and \ref{thm:petersen}}\label{sec:petersen}

We now proceed to the proofs of Theorems \ref{thm:9-vtx-c2} and \ref{thm:petersen}, but first introduce notation for the state of the game. We fix a
connected graph $G$ on which the game is played. The state of the game
is a pair $\st{\bc}{r}$, where $G$ is a connected graph, $\bc$ is a
$k$-tuple of vertices $\bc = (c_1, c_2, \dots, c_k)$, where $c_i\in
V(G)$ is the current position of cop $C_i$, and $r\in V(G)$ is the
current position of the robber $R$.  For notational convenience, we
write $\st{c_1, \dots, c_k}{r}$ for $\st{(c_1, \dots, c_k)}{r}$. When
we need to specify whose turn it is to act, we underline the position
of the player whose turn it is: $\copst{\bc}{r}$ denotes that it is
the cops' turn to move, and $\robst{\bc}{r}$ the robber's.

We use a shorthand notation to describe moves: $\copst{c_1, \dots,
  c_k}{r}\copmove \robst{c'_1, \dots, c'_k}{r}$ denotes the cop
move where each $C_i$ moves from $c_i$ to $c'_i$. Similarly
$\robst{c_1, \dots, c_k}{r}\robmove \copst{c_1, \dots, c_k}{r'}$
denotes the robber's move from $r$ to $r'$. We will concatenate moves
and we use the shorthand $\crmove$, meaning a cop move followed by a
robber move:
$$\copst{c_1, \dots, c_k}{r}\crmove \copst{c'_1, \dots, c'_k}{r'}$$
is equivalent to
$$\copst{c_1, \dots, c_k}{r}\copmove
\robst{c'_1, \dots, c'_k}{r}\robmove\copst{c'_1, \dots,
  c'_k}{r'}. $$
There will be cases where the strategy allows for
either the robber or the cops to be in one of several positions. In
general, for $T_i\subseteq V$, $S\subseteq V$, the state of the game
has the form $\st{T_1, \dots, T_k}{S}$ means that $c_i\in T_i$, and
$r\in S$.

The robber's \emph{safe neighborhood}, denoted $S(R)$, is the
connected component of $G-N[C]$
containing the robber. We say that the robber is \emph{trapped} when
$S(R) = \emptyset$. This condition is equivalent to having both $r \in
N(\bc)$ and $N(r) \subseteq N[\bc]$. Once the robber is trapped, he will
be caught on the subsequent cop move, regardless of the robber's next
action. When the robber is trapped, we are in a \emph{cop-winning position}, denoted by
$\copwin$.

\subsection{The end game}

We frequently use the following facts to identify cop-win
strategies for two cops in the end game.
We state a more general version of these results
for $k$ cops.

We need the following property arising in the study of cop-win graphs, which first appears in \cite{cn}. Every cop-win graph has at least one winning \emph{no-backtrack strategy} for the cop, meaning a winning strategy where the cop never repeats a vertex during the pursuit. Typically, a graph has multiple winning no-backtrack strategies. We say that a vertex $v$ is \emph{no-backtrack-winning} if there is a winning no-backtrack strategy for the cop starting at $v$. For example, when $G$ is a tree, every vertex is no-backtrack-winning.

Next we fix some notation. For a set $U = \{ u_1, \dots, u_t \}
\subseteq V$, let $N_U'(u_j) = N(u_j) \setminus N[U \setminus u_{j} ]$ be the
neighbors of $u_j$ that are not adjacent to any other vertex in $U$.

\begin{lemma}\label{obs:key}
Let $\copst{\bc}{r}$ be the state of the game.  Suppose that  there exists a $c_j \in \bc$ such that either
\begin{inparaenum}[\upshape(\itshape a\upshape)]
\item \label{key:empty} $[S(R): N'_C(c_j)] = \emptyset$ and $G[S(R)]$ is cop-win; or
\item \label{key:one}
$N(S(R)) \cap   N'_C(c_j) = \{ v \}$ such that $H=G[S(R) \cup \{v\}]$ is cop-win and $v$
is no-backtrack-winning in $H$.
\end{inparaenum}
  Then the cops can win from this configuration.
\end{lemma}

\begin{proof}
Let $S=S(R)$ be the initial safe neighborhood of $R$. In both cases, only cop $C_j$ is active, while the others remain stationary.   In case (a), cop $C_j$ moves into $S$ and follows a cop-win strategy on $G[S]$.
In case (b),  cop $C_j$ moves  to $v$ and then follows a no-backtrack strategy on $G[S \cup \{v\}]$.  This prevents the robber from ever reaching $v$. In both cases, the only way for the robber to avoid capture by $C_j$ is to move into the neighborhood of the remaining cops.
\end{proof}

We highlight two useful consequences that are  used heavily for $k=2$ in subsequent proofs.

\begin{corollary}\label{cor:small-safe-nbhd}
Let $\copst{\bc}{r}$ be the state of the game, played with $k \geq 2$
cops. If $|S(R)|\leq 2$ and $|N(S(R))|\leq 2k-1$, then the cops can
win.
\end{corollary}
\begin{proof}
Let $S=S(R)$.
  We have  $|N(S)\cap
  N(\bc)|\leq 2k-1$, so the pigeonhole principle ensures that  there exists a cop
  $C_j$ such that $|N(S)\cap N'_C(c_j)|\leq 1$. We are done by Lemma~\ref{obs:key}, since every vertex of a connected $2$-vertex graph is non-backtrack-winning.
%  If $|S : N'_C(c_j)| \leq 1$ then
%  conditions of Observation~\ref{obs:key}(\ref{key:one}) are satisfied.
%    Otherwise, let $v$ be the unique vertex in $N(S)\cap N'_C(c_j)$. If $v$ is adjacent to all 3 vertices in $S$, then moving $C_j$ to $v$ traps the robber.
%
%  Suppose that $v$ is adjacent to 2 of the 3 vertices in $S$. Once again, we move $C_j$ to $v$. The robber is trapped unless $G[S]$ is a path $u_1, u_2, u_3$ with $v \sim u_1$ and $v \sim u_2$. In this case, the robber must respond by moving to (or staying at) $u_2$. If $\deg_G(u_2)=2$, then $C_j$ stays at $v$ while another cop moves in to capture the robber.
%  Otherwise $u_2$ is adjacent to $w \in N(c_i)$ for some $i$. With the remaining cops holding their positions, the game ends as
%  $\copst{v,c_i}{u_2}  \copmove \robst{u_1, c_i}{ u_2} \robmove \copst{u_1, c_i}{u_3}
%  \copmove \robst{v, w}{u_3} \copwin.$
\end{proof}

\begin{corollary}\label{cor:one-deg-3}
Let $\copst{\bc}{r}$ be the state of the game, played with $k \geq 2$ cops. If
$$\max_{v \in S(R)} \deg_G(v) \leq 3$$ and $S(R)$
 contains at most one vertex of degree $3$, then the cops can win. \endproof
\end{corollary}

\begin{proof}
Let $S=S(R)$. Since $G[S]$ is connected, we have $[S : N(C)] \leq 3$. Therefore, some cop $C_j$ has $|S: N'_C(c_j)| \leq 1$. If $G[S]$ is a tree, then we are done by Lemma~\ref{obs:key}.
If $G[S]$ is not a tree, then $G[S]$ must be unicyclic with one degree 3 vertex, say $u$. Therefore,
$|S: N(C)|=1$, and except for $u$, every vertex in the cycle has degree 2 in $G$. A winning strategy for the cops is as follows:  two cops move until they both reach $u$. Now $S(R)$ is a path, so  Lemma~\ref{obs:key} completes the proof.
\end{proof}

%%%%%%%%%%%%%%%%%%%%%%%%%%%%%%%%%%%%%%%%%%%%%%%%%%%%%%%%%%%%%%%%%%%%%%%%%%%%%%%%%%%%%
%%%%%%%%%%%%%%%%%%%%%%%%%%%%%%%%%%%%%%%%%%%%%%%%%%%%%%%%%%%%%%%%%%%%%%%%%%%%%%%%%%%%%
%%%%%%%%%%%%%%%%%%%%%%%%%%%%%%%%%%%%%%%%%%%%%%%%%%%%%%%%%%%%%%%%%%%%%%%%%%%%%%%%%%%%%

\subsection{Graphs with $\Delta(G) \geq n-6$}

In this section, we prove  Theorem \ref{thm:9-vtx-c2}. We also make progress on the proof of Theorem \ref{thm:petersen} by showing that if $v(G)=10$ and $\Delta(G)=4$, then $c(G)\leq 2$. For convenience, we recall the statements of these results prior to their respective proofs.

\begin{lemma}\label{lemma:gen-5-cycle}
Let $G$ be a graph on $n$ vertices. If there is a vertex $u\in
V(G)$ of degree at least $n-6$, then either $c(G)\leq 2$ or the induced
subgraph $G[V\setminus N[u]]$ is a 5-cycle.
\end{lemma}

\begin{corollary}
\label{cor:deg=n-5}
If $\Delta(G) \geq n-5$, then $c(G) \leq 2$. \endproof
\end{corollary}

Lemma \ref{lemma:gen-5-cycle} and its immediate corollary are crucial
tools in proving the main results. In particular,
Theorem~\ref{thm:9-vtx-c2} is a quick consequence of Corollary
\ref{cor:deg=n-5}. This reduces the search to 10 vertex graphs with $2
\leq \delta(G) \leq \Delta(G) \leq 4$.

\begin{proof}[Proof of Lemma~\ref{lemma:gen-5-cycle}]
Let $H=G[V\setminus N[u]]$. By Lemma~\ref{obs:key}(\ref{key:empty}), if
$H$ is cop-win, then $c(G)\leq 2$. In particular this holds if $H$
does not contain an induced cycle of length at least 4. So we only
need to consider the case where $v(H) = 4$ or $5$, and $H$ contains an
induced $4$-cycle.  Let $x_1, x_2, x_3, x_4$ form the 4-cycle in $H$
(in that order). Let $x_5$ be the additional vertex (if present).

We now distinguish some cases based on $N(x_5)\cap H$. If $x_5\sim x_i$
for every $i\in\{1, 2, 3, 4\}$, then $H$ is cop-win, and hence,
$c(G)\leq 2$. We therefore have 5 cases to consider, depicted in
Figure~\ref{fig:4-cycle-cases}. Case (a) includes the situation when
$\deg(u)=n-5$, and there is no vertex $x_5$.

\begin{figure}[ht]
\begin{center}
\begin{tabular}{ccccc}

%%%%%%%%%%%%%%%%
% FIRST PICTURE
\begin{tikzpicture}[scale=.85]

\path(0,0) coordinate (X3);
\path(1,1) coordinate (X2);
\path(1,-1) coordinate (X4);
\path(2,0) coordinate (X1);
\path(1,0) coordinate (X5);

\foreach \i in {1,2,3,4,5}
{
\draw[fill] (X\i) circle (2pt);
}

\node at (2.25,-.25) {$x_1$};
\node[above] at (X2) {$x_2$};
\node at (-.2,-.25) {$x_3$};
\node[below] at (X4) {$x_4$};
\node at (.75,-.25) {$x_5$};

\draw (X1) -- (X2) -- (X3) -- (X4) -- (X1);

\end{tikzpicture}

&
%%%%%%%%%%%%%%%%%%%
%% SECOND PICTURE
\begin{tikzpicture}[scale=.85]

\path(0,0) coordinate (X3);
\path(1,1) coordinate (X2);
\path(1,-1) coordinate (X4);
\path(2,0) coordinate (X1);
\path(1,0) coordinate (X5);

\foreach \i in {1,2,3,4,5}
{
\draw[fill] (X\i) circle (2pt);
}

\node at (2.25,-.25) {$x_1$};
\node[above] at (X2) {$x_2$};
\node at (-.2,-.25) {$x_3$};
\node[below] at (X4) {$x_4$};
\node at (.75,-.25) {$x_5$};

\draw (X1) -- (X2) -- (X3) -- (X4) -- (X1);

\draw (X2) -- (X5);

\end{tikzpicture}

&
%%%%%%%%%%%%%%%%%%%
%% THIRD PICTURE
\begin{tikzpicture}[scale=.85]

\path(0,0) coordinate (X3);
\path(1,1) coordinate (X2);
\path(1,-1) coordinate (X4);
\path(2,0) coordinate (X1);
\path(1,0) coordinate (X5);

\foreach \i in {1,2,3,4,5}
{
\draw[fill] (X\i) circle (2pt);
}

\node at (2.25,-.25) {$x_1$};
\node[above] at (X2) {$x_2$};
\node at (-.2,-.25) {$x_3$};
\node[below] at (X4) {$x_4$};
\node at (.75,-.25) {$x_5$};

\draw (X1) -- (X2) -- (X3) -- (X4) -- (X1);

\draw (X1) -- (X5);
\draw (X2) -- (X5);

\end{tikzpicture}
&

%%%%%%%%%%%%%%%%%%%
%% FOURTH PICTURE
\begin{tikzpicture}[scale=.85]

\path(0,0) coordinate (X3);
\path(1,1) coordinate (X2);
\path(1,-1) coordinate (X4);
\path(2,0) coordinate (X1);
\path(1,0) coordinate (X5);

\foreach \i in {1,2,3,4,5}
{
\draw[fill] (X\i) circle (2pt);
}

\node at (2.25,-.25) {$x_1$};
\node[above] at (X2) {$x_2$};
\node at (-.2,-.25) {$x_3$};
\node[below] at (X4) {$x_4$};
\node at (.75,-.25) {$x_5$};

\draw (X1) -- (X2) -- (X3) -- (X4) -- (X1);

\draw (X2) -- (X4);

\end{tikzpicture}
&
%%%%%%%%%%%%%%%%%%%
%% FIFTH PICTURE
\begin{tikzpicture}[scale=.85]

\path(0,0) coordinate (X3);
\path(1,1) coordinate (X2);
\path(1,-1) coordinate (X4);
\path(2,0) coordinate (X1);
\path(1,0) coordinate (X5);

\foreach \i in {1,2,3,4,5}
{
\draw[fill] (X\i) circle (2pt);
}

\node at (2.25,-.25) {$x_1$};
\node[above] at (X2) {$x_2$};
\node at (-.2,-.25) {$x_3$};
\node[below] at (X4) {$x_4$};
\node at (.75,-.25) {$x_5$};

\draw (X1) -- (X2) -- (X3) -- (X4) -- (X1);

\draw (X1) -- (X5);
\draw (X2) -- (X5);
\draw (X4) -- (X5);

\end{tikzpicture}

\\
(a) & (b) & (c) & (d) & (e)

\end{tabular}

\caption{The five cases for $G[V \setminus N[u]]$ in  Lemma~\ref{lemma:gen-5-cycle}.}

\label{fig:4-cycle-cases}
\end{center}
\end{figure}
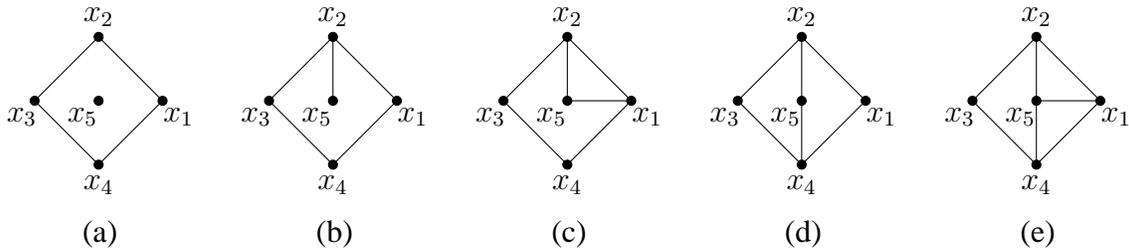

First we make some technical claims. We start by noting that
moving to $x_5$ is in most situations a bad idea for the robber in
Cases (a), (b) and (c).

\smallskip

\noindent \textbf{Claim 1}. In Cases (a), (b), and (c), if the state of the game is of the form
$\copst{N[u], V(H)}{x_5}$, the cops have a winning strategy.
%\label{claim:Rx5bad}

\smallskip

For the proof of the claim, $C_1$ moves to $u$. In Case (a), $S(R) = \{x_5\}$ and we are already
  done by Corollary~\ref{cor:small-safe-nbhd}. In Case (b), if possible,
  $C_2$ moves directly to $x_2$; otherwise, $C_2$ moves first to $x_1$
  and then to $x_2$; in either case the robber is trapped at $x_5$. In
  Case (c), $C_2$ moves to $x_2$ or $x_1$ (whichever $c_2$ is adjacent
  to), again trapping the robber in $x_5$ The proof of the claim follows.

Next we consider the structure of $N(y)\cap V(H)$ for vertices $y\in
N(u)$.

\smallskip

\noindent \textbf{Claim 2}.
%\label{claim:good-state}
Suppose the state of the game has the form $\copst{N[u], \{x_1,
  x_3\}}{y}$, where $y\in N(u)$ is such that
either \begin{inparaenum}[\upshape(\itshape a\upshape)]
\item \label{case:both} $N(y)\cap V(H)
  = \{x_2, x_4\}$, or \item \label{case:one}$y$ is adjacent to at most one of $x_2$ or
  $x_4$; \end{inparaenum} then the cops have a winning strategy.

\smallskip

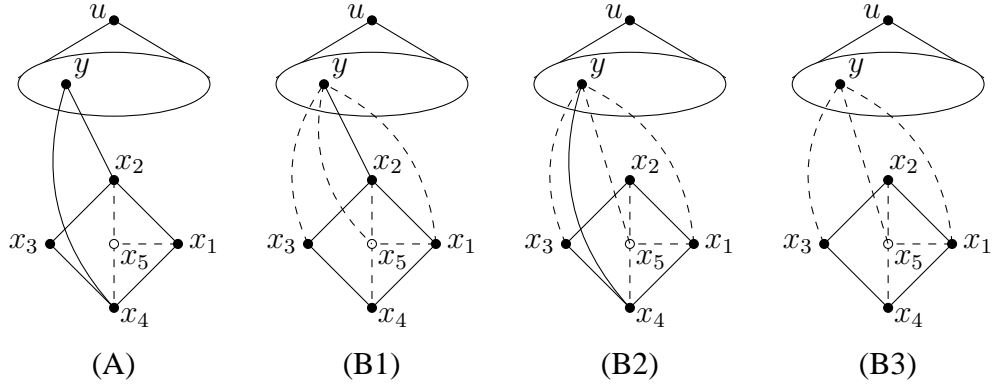
\begin{figure}
\begin{center}
%%%%%%%%%%%%%%%%%%%
%% more 4 cycle endgame
\begin{tabular}{cccc}
\begin{tikzpicture}[scale=.85]

%%%%%%%%%%%%%%
% V - \cN(u)

\path(0,0) coordinate (X3);
\path(1,1) coordinate (X2);
\path(1,-1) coordinate (X4);
\path(2,0) coordinate (X1);
\path(1,0) coordinate (X5);

\foreach \i in {1,2,3,4}
{
\draw[fill] (X\i) circle (2pt);
}

\node[right] at (X1) {$x_1$};
\node at (1.25,1.25) {$x_2$};
\node[left] at (X3) {$x_3$};
\node at (1.33,-1.15) {$x_4$};
\node at (1.33,-.25) {$x_5$};

\draw (X1) -- (X2) -- (X3) -- (X4) -- (X1);

\draw[dashed] (X2) -- (X5) -- (X4);
\draw[dashed] (X1) -- (X5);

%% this vertex may or may not be present
\draw[fill=white] (X5) circle (2pt);

%%%%%%
% \cN(u)

\path(1,3.5) coordinate (U);
\path(.25,2.5) coordinate (Y);

\draw (-.5, 2.6) -- (U) -- (2.5, 2.6);

\draw[fill=white] (1,2.5) ellipse (1.5 and .5);

\draw[fill] (U) circle (2pt);
\draw[fill] (Y) circle (2pt);

\node at (.75,3.65) {$u$};
\node at (.5,2.75) {$y$};

\draw (Y) -- (X2);
\draw (Y) to [bend right] (X4);

\end{tikzpicture}

&
\begin{tikzpicture}[scale=.85]

%%%%%%%%%%%%%%
% V - \cN(u)

\path(0,0) coordinate (X3);
\path(1,1) coordinate (X2);
\path(1,-1) coordinate (X4);
\path(2,0) coordinate (X1);
\path(1,0) coordinate (X5);

\foreach \i in {1,2,3,4}
{
\draw[fill] (X\i) circle (2pt);
}

\node[right] at (X1) {$x_1$};
\node at (1.25,1.25) {$x_2$};
\node[left] at (X3) {$x_3$};
\node at (1.33,-1.15) {$x_4$};
\node at (1.33,-.25) {$x_5$};

\draw (X1) -- (X2) -- (X3) -- (X4) -- (X1);

\draw[dashed] (X2) -- (X5) -- (X4);
\draw[dashed] (X1) -- (X5);

%% this vertex may or may not be present
\draw[fill=white] (X5) circle (2pt);

%%%%%%
% \cN(u)

\path(1,3.5) coordinate (U);
\path(.25,2.5) coordinate (Y);

\draw (-.5, 2.6) -- (U) -- (2.5, 2.6);

\draw[fill=white] (1,2.5) ellipse (1.5 and .5);

\draw[fill] (U) circle (2pt);
\draw[fill] (Y) circle (2pt);

\node at (.75,3.65) {$u$};
\node at (.5,2.75) {$y$};

\draw (Y) -- (X2);
\draw[dashed] (Y) to [bend right] (X5);
\draw[dashed] (X1) to [bend right] (Y);
\draw[dashed] (X3) to [bend left] (Y);

\end{tikzpicture}
&

\begin{tikzpicture}[scale=.85]

%%%%%%%%%%%%%%
% V - \cN(u)

\path(0,0) coordinate (X3);
\path(1,1) coordinate (X2);
\path(1,-1) coordinate (X4);
\path(2,0) coordinate (X1);
\path(1,0) coordinate (X5);

\foreach \i in {1,2,3,4}
{
\draw[fill] (X\i) circle (2pt);
}

\node[right] at (X1) {$x_1$};
\node at (1.25,1.25) {$x_2$};
\node[left] at (X3) {$x_3$};
\node at (1.33,-1.15) {$x_4$};
\node at (1.33,-.25) {$x_5$};

\draw (X1) -- (X2) -- (X3) -- (X4) -- (X1);

\draw[dashed] (X2) -- (X5) -- (X4);
\draw[dashed] (X1) -- (X5);

%% this vertex may or may not be present
\draw[fill=white] (X5) circle (2pt);

%%%%%%
% \cN(u)

\path(1,3.5) coordinate (U);
\path(.25,2.5) coordinate (Y);

\draw (-.5, 2.6) -- (U) -- (2.5, 2.6);

\draw[fill=white] (1,2.5) ellipse (1.5 and .5);

\draw[fill] (U) circle (2pt);
\draw[fill] (Y) circle (2pt);

\node at (.75,3.65) {$u$};
\node at (.5,2.75) {$y$};

\draw (Y) to [bend right] (X4);
\draw[dashed] (Y) to  (X5);
\draw[dashed] (X1) to [bend right] (Y);
\draw[dashed] (X3) to [bend left] (Y);

\end{tikzpicture}
&

\begin{tikzpicture}[scale=.85]

%%%%%%%%%%%%%%
% V - \cN(u)

\path(0,0) coordinate (X3);
\path(1,1) coordinate (X2);
\path(1,-1) coordinate (X4);
\path(2,0) coordinate (X1);
\path(1,0) coordinate (X5);

\foreach \i in {1,2,3,4}
{
\draw[fill] (X\i) circle (2pt);
}

\node[right] at (X1) {$x_1$};
\node at (1.25,1.25) {$x_2$};
\node[left] at (X3) {$x_3$};
\node at (1.33,-1.15) {$x_4$};
\node at (1.33,-.25) {$x_5$};

\draw (X1) -- (X2) -- (X3) -- (X4) -- (X1);

\draw[dashed] (X2) -- (X5) -- (X4);
\draw[dashed] (X1) -- (X5);

%% this vertex may or may not be present
\draw[fill=white] (X5) circle (2pt);

%%%%%%
% \cN(u)

\path(1,3.5) coordinate (U);
\path(.25,2.5) coordinate (Y);

\draw (-.5, 2.6) -- (U) -- (2.5, 2.6);

\draw[fill=white] (1,2.5) ellipse (1.5 and .5);

\draw[fill] (U) circle (2pt);
\draw[fill] (Y) circle (2pt);

\node at (.75,3.65) {$u$};
\node at (.5,2.75) {$y$};

\draw[dashed] (Y) to  (X5);
\draw[dashed] (X1) to [bend right] (Y);
\draw[dashed] (X3) to [bend left] (Y);

\end{tikzpicture}
\\
(A) & (B1) & (B2) & (B3)
\end{tabular}

\caption{The four classes of possible structures of $G$ for Claim~2. Vertex $x_5$  might not be present, and dashed edges might not be present.}
\label{fig:4cycle-continued}

\end{center}
\end{figure}

For the proof of the claim, Figure \ref{fig:4cycle-continued} shows the four classes of possible graph structures.
Let us first consider the structure (B1). Let $z=x_2$.
$C_1$ moves to $u$, and $C_2$ moves
to $z$. Now the robber is trapped in all cases of Figure \ref{fig:4-cycle-cases} except Case (a). In Case (a) the robber's only move is to $x_5$. After this move,  the robber can be caught by Claim~1.
The same cop strategy works for structures (B2) and (B3), taking $z=x_4$. A simplified version of this proof shows that the same cop strategy works for structure (A), taking $z=c_2$. The proof of the claim follows.

We remark that in Cases (d) and (e) of Figure \ref{fig:4-cycle-cases}, $x_5$ and $x_1$ are
symmetric, so the statement holds also for configuration $\copst{N[u], x_5}{y}$

The next claim concerns the situation where there are two vertices in $N(u)$
that do not satisfy the condition of the previous claim.

\smallskip

\noindent \textbf{Claim 3}.
%\label{claim:2-dominating}
If there are two vertices $y, z\in N(u)$ such that $ \{ x_2, x_3,x_4 \} \subseteq N(y)$, and
$\{ x_1, x_2, x_4 \} \subseteq N(z)$, then $c(G)\leq 2$.

\smallskip

For the proof of Claim 3, first we deal with all cases but Case (d). The cops start at $u$ and
$z$. If the robber starts at $x_3$, the cops' winning strategy is:
$\copst{u, z}{x_3}\copmove \robst{u, y}{x_3}\copwin$. \textbf{[Anthony: clarify the use of the $\copwin$ notation here.]}  If the robber
starts at $x_5$, the strategy will depend on the structure of $H$. In
Cases (a), (b), and (c) we are done by Claim~1. In Case
(e) the following is a winning strategy: $\copst{u, z}{x_5}\copmove
\robst{u, x_1}{x_5}\copwin$.

The remainder of the proof deals with Case (d), which requires a more
involved argument.

First suppose that there exists $w\in N(u)$ such that $\{x_2, x_4, x_5\}\subseteq N(w)$. Then the cops start at $u$ and $z$. The robber can
start at $x_3$ or $x_5$ in either case the cops have a winning
strategy: $\copst{u, z}{x_3}\copmove \robst{y,
  u}{x_3}\copwin$; or $\copst{u, z}{x_5}\copmove \robst{w,
  u}{x_5}\copwin$.

Now assume that no such $w$ exists. Start the cops at $u$ and $y$. The
robber starts in  $\{ x_1, x_5\}$. If the robber starts at $x_1$,
then $\copst{u, y}{x_1}\copmove\robst{y, u}{x_1}\copwin$. So we may assume the
robber starts at $x_5$. If $|N(x_5)\cap N(u)|\leq 1$, we are done by
Corollary~\ref{cor:small-safe-nbhd}. Otherwise, the cops move by $\copst{u,
x_3}{x_5}\copmove \robst{v, z}{x_5}$, for some $v\in N(x_5)\cap N(u)$. The
robber is forced to move to some $w\in N(x_5)\cap N(u)$ (if no such
$w$ exists, then $R$ is trapped). By our initial argument, $w$ cannot
be adjacent to both $x_2$ and $x_4$, so the state satisfies
the conditions of Claim~2(\ref{case:one}), and the proof of the claim follows.

% The general strategy is as follows: Start the cops at $u$ and
% $x_1$. We show that the robber must start at $x_3$. Then we move $C_1$
% to some $v\in N(u)\cap N(x_3)$. The robber is forced to some $y\in
% N(u)$. If $y$ does not satisfy the conditions of
% Claim~\ref{claim:good-state}, then we can force the robber to some
% $z\in N(u)$. If $z$ does not satisfy the conditions of
% Claim~\ref{claim:good-state}, then the pair $y, z$ will satisfy the
% conditions of Claim~\ref{claim:2-dominating}.
%
% Now we describe the cops' strategy in detail.

\smallskip

\noindent \textbf{Claim 4}.
%\label{claim:x1-N(u)}
Either $c(G)\leq 2$, or we can relabel the vertices of $H$ via an
automorphism of $H$ so that $x_1$ is adjacent to $N(u)$.

\smallskip

To prove the claim, suppose that no such relabeling exists. We will show a
winning strategy for the cops, starting at $u$ and $x_3$. In Cases (a)
and (b) the claim follows from Corollary~\ref{cor:small-safe-nbhd}
(either $S(R) = \{x_1\}$ or $S(R) = \{x_5\}$). In Cases (d) and (e), $S(R)\subseteq\{x_1, x_5\}$,
and we are assuming that both $x_1$ and $x_5$ have no edges to $N(u)$;
hence, $|N(S(R))|\leq 2$, and we are again done by
Corollary~\ref{cor:small-safe-nbhd}. In Case (c) $S(R)=\{x_1, x_5\}$,
and we are assuming that both $x_1$ and $x_2$ do not have neighbors in
$N(u)$. By Claim~1, we may assume $R$ does not start
at $x_5$, and so $R$ starts at $x_1$. Let $v\in N(u)\cap N(x_5)$ (if
$x_5\nsim N(u)$, then $N(S(R))$ is dominated by $c_2=x_3$). Now the
cops can win by following the strategy: $\copst{u, x_3}{x_1}\copmove
\copst{v, x_3}{x_1}\copwin$. The proof of the claim follows.

\medskip

Armed with the above claims, we now conclude the proof
Lemma~\ref{lemma:gen-5-cycle}.  By Claim~4, we may assume $x_1\sim w\in N(u)$. Initially
place $C_1$ at $u$ and $C_2$ at $x_1$. The robber could start at $x_3$
or, in Cases (a), (b), and (d), at $x_5$. If the robber starts at
$x_5$ in Cases (a) and (b), then we are done by
Claim~1. In Case (d), $x_5$ and $x_3$ are symmetric,
so without loss of generality, $r = x_3$, and the initial state is
$\copst{u, x_1}{x_3}$.

If $x_3\nsim N(u)$, then the cops win by Corollary~\ref{cor:small-safe-nbhd}.
Otherwise let $v\in N(x_3)\cap N(u)$. Then $C_1$ moves from $u$ to
$v$, while $C_2$ remains fixed at $x_1$, forcing $R$ to some $y\in
N(u)\cap N(x_3)$, with $y\nsim v, y\nsim x_1$.  If no such $y$ exists, then $R$ is
trapped. If $y$ is adjacent to only one of $x_2$ or $x_4$, we are in
the state $\copst{v\in N(u), x_1}{y}$, which satisfies the conditions of
Claim~2 (\ref{case:one}), and hence, the cops have
a winning strategy.

Otherwise $y$ is adjacent to $x_2, x_3$, and $x_4$. The cops move
$\copst{v, x_1}{y}\copmove \robst{x_3, w}{y}$, for some $w\in
N(x_1)\cap N(u)$. If $y\sim x_5$, and $R$ moves to $x_5$, then the
cops win: in Cases (a),(b),(c) we are done by
Claim~1; in Case (d), $\copst{x_3, w}{x_5}\copmove
\robst{y, u}{x_5}\copwin$; in Case (e), the cops can adopt a different
strategy from the beginning: $\copst{u, y}{x_1}\copmove \robst{u,
  x_5}{x_1}\copwin$. The only other option is for $R$ to move to some
$z\in N(u)$, $z\nsim x_3$. So the state is $\copst{x_3, w}{z}$. Either
the pair $y, z$ satisfies the conditions of
Claim~3, or the current state satisfies the
conditions of Claim~2(\ref{case:one}) or
(\ref{case:both}).  In either case, we are done. This concludes the
proof of Lemma \ref{lemma:gen-5-cycle}.
\end{proof}

We now state some quick but useful consequences of
Lemma~\ref{lemma:gen-5-cycle}.

\begin{corollary}
\label{cor:deg4b}
Let $G$ be a graph on $n$ vertices. If there is a vertex $u\in V$ of
degree at least $n-6$, and a vertex $v\in V\setminus N[u]$ such that
$|N(v)\setminus N(u)|\geq 3$, then $c(G) \leq 2$.  \endproof
\end{corollary}
\begin{proof}
The vertex $v$ has three neighbors in $G[V\setminus N[u]]$, and hence, $G[V\setminus N[u]]$ cannot be a 5-cycle.
\end{proof}

\smallskip

\begin{corollary}~\label{cor:deg4deg3}
Let $G$ be a graph on $n$ vertices. If there is a vertex $u$ of degree
at least $n-6$ and a vertex  $v \in V \setminus N[u]$ with $\deg(v) \leq 3$, then
$c(G)\leq 2$.
\end{corollary}

\begin{proof}
By Lemma~\ref{lemma:gen-5-cycle}, we only need to consider the case
where $G[V \setminus N[u]]$ is a 5-cycle, $x_1, x_2, x_3, x_4, x_5$ (in that order). Without
loss of generality, let $\deg(x_1)\leq 3$, and $\deg(x_2)\geq 3$. For
each $i=1, \dots, 5$ such that $\deg(x_i)\geq 3$, pick some $y_i\in
N(x_i)\cap N(u)$ arbitrarily (we allow $y_i=y_j$ for $i\neq
j$). The game starts as $\copst{u,x_4}{\{x_1,x_2\}} \crmove \copst{u, \{x_3, x_4\}}{x_1}$.
First we deal with the case where $\deg(x_1)=2$ and the case where
$\deg(x_1)=3$ and $y_1\sim x_4$.
The cops' winning strategy for these two cases is the same: $\copst{u,
  \{x_3, x_4\}}{x_1}\crmove \copst{y_2, x_4}{x_1}\copmove \robst{x_2,
  x_4}{x_1}\copwin.$

Now we may assume that all $x_i$ have degree $3$, and hence, $y_i$ exists for
all $i$. We may further assume that $x_4\neq y_1$, and, since $x_3$ and $x_4$ are symmetric, we are also done in the case $y_1\sim x_3$.
The only remaining possibility is $N(y_1)\cap
(V \setminus N[u])\subseteq N[x_1]$.  Since $x_3$ and $x_4$ are symmetric,
without loss of generality, the state is $\copst{u, x_4}{x_1}$. The
cops first move to $y_2$ and $x_5$, forcing the robber to $y_1$, then
in one more move, the robber is trapped at $y_1$: $\copst{y_2,
  x_5}{y_1}\copmove\robst{u, x_1}{y_1}\copwin.$
\end{proof}

These corollaries are enough to prove that every 9-vertex
graphs is 2 cop-win, and to show that if $v(G)=10$ and $\Delta(G)=4$
then $c(G) \leq 2$.

\begin{proof}[Proof of Theorem \ref{thm:9-vtx-c2}]
If $\Delta(G)\geq 4$, then we are done by
Lemma~\ref{lemma:gen-5-cycle}. If $\Delta(G) = 3$, then we are done by
Corollary~\ref{cor:deg4deg3}.
\end{proof}

\begin{lemma}\label{lemma:deg-leq-3}
If $v(G)=10$ and $\Delta(G)\geq 4$, then $c(G)\leq 2$.
\end{lemma}
\begin{proof}
Let $u\in V(G)$ have degree at least $4$. By
Lemma~\ref{lemma:gen-5-cycle}, either $c(G)\leq 2$ or $\deg(u)=4$, and
$G[V \setminus N[u]]$ is a 5-cycle. Now, by Corollary~\ref{cor:deg4deg3},
either $c(G)\leq 2$, or every $u\in V - N[u]$ has $\deg(u)\geq 4$. In
the latter case, $|[N(u):V\setminus N[u]]|\geq 10$; thus, by the pigeonhole
principle, there exists $v\in N(u)$ such that $|N(v)\cap (V\setminus N[u])|\geq
3$. We now deal with this case, namely $u$ and $v$ have degree $4$,
and $N(u)\cap N(v) = \emptyset$.

By Lemma~\ref{lemma:gen-5-cycle}, both $G[V(G)\setminus N[u]]$ and
$G[V(G)-\setminus N[v]]$ are $5$-cycles. The resulting graph
structure must be one of the two shown in Figure \ref{fig:deg4c}.
Considering the structure in Figure \ref{fig:deg4c}(a), we note that
$\deg(z_1) = \deg(z_2)=3$ in order to maintain the induced 5-cycle
structures, and hence, we are done by Corollary~\ref{cor:deg4deg3}.

Now suppose that $G$ has the structure in Figure
\ref{fig:deg4c}(b). In this case we show $\deg(x_3)=3$, and we are
again done by Corollary~\ref{cor:deg4deg3}. To show that
$\deg(x_3)=3$, we look at each potential additional edge, and show that
$V \setminus N[x_3]$ is not a 5-cycle, and hence, we are done by
Lemma~\ref{lemma:gen-5-cycle}. We only need to consider edges to $y_1,
y_2$ or $y_3$: other potential edges would not maintain the induced
5-cycle structure. We have $x_3 \nsim y_1$ because $\{v, y_2, y_3 \}$
form a triangle. We have $x_3 \nsim y_3$ because $z_1$ is adjacent to
each of $x_1, y_1, y_2$. Finally, $x_3 \nsim y_2$ because the
existence of this edge would force $y_3 \sim x_1$, which is symmetric
to the forbidden $x_3 \sim y_1$.
\end{proof}

\begin{figure}[ht]
\begin{center}
\begin{tabular}{cc}

%%%%%%%%%%%%%%
% symmetric v1
\begin{tikzpicture}[scale=.8]

\path (0,0) coordinate (U);
\path (4,0) coordinate (V);

\path (2,-.5) coordinate (Z2);
\path (2,.5) coordinate (Z1);

\path (1,-1) coordinate (X3);
\path (1,0) coordinate (X2);
\path (1,1) coordinate (X1);

\path (3,-1) coordinate(Y3);
\path (3,0) coordinate (Y2);
\path (3, 1) coordinate (Y1);

\draw[fill] (U) circle (2pt);
\draw[fill] (V) circle (2pt);

\draw (U) circle (4pt);
\draw (V) circle (4pt);

\node[left=2pt] at (U) {$u$};
\node[right=2pt] at (V) {$v$};

\foreach \i in {1,2,3}
{
\draw[fill] (Y\i) circle (2pt);
\draw (V) -- (Y\i);

\draw[fill] (X\i) circle (2pt);
\draw (U) -- (X\i);
}

\node[above] at (X1) {$x_1$};
\node[above] at (Y1) {$y_1$};

\node[right] at (X2) {$x_2$};
\node[left] at (Y2) {$y_2$};

\node[below] at (X3) {$x_3$};
\node[below] at (Y3) {$y_3$};

\foreach \i in {1, 2}
{
\draw[fill] (Z\i) circle (2pt);
\draw (Z\i) circle (4pt);
}

\node[above=2pt] at (Z1) {$z_1$};
\node[below=2pt] at (Z2) {$z_2$};
\draw (Z1) -- (X1) -- (X2) -- (X3) -- (Z2) -- (Y3) -- (Y2) -- (Y1) -- (Z1) -- (Z2);

\path (0,2.5) coordinate (CU);
\path (4,2.5) coordinate (CV);
\draw (U) .. controls (CU) and (CV) .. (V);

\end{tikzpicture}

&
%%%%%%%%%%%%%%%%%%%
% symmetric v2
\begin{tikzpicture}[scale=.8]

\path (0,0) coordinate (U);
\path (4,0) coordinate (V);

\path (2,-.5) coordinate (Z2);
\path (2,.5) coordinate (Z1);

\path (1,-1) coordinate (X3);
\path (1,0) coordinate (X2);
\path (1,1) coordinate (X1);

\path (3,-1) coordinate(Y3);
\path (3,0) coordinate (Y2);
\path (3, 1) coordinate (Y1);

\draw[fill] (U) circle (2pt);
\draw[fill] (V) circle (2pt);

\draw (U) circle (4pt);
\draw (V) circle (4pt);

\node[left=2pt] at (U) {$u$};
\node[right=2pt] at (V) {$v$};

\foreach \i in {1,2,3}
{
\draw[fill] (Y\i) circle (2pt);
\draw (V) -- (Y\i);

\draw[fill] (X\i) circle (2pt);
\draw (U) -- (X\i);
}

\node[above] at (X1) {$x_1$};
\node[above] at (Y1) {$y_1$};

\node[above] at (X2) {$x_2$};
\node[above] at (Y2) {$y_2$};

\node[below] at (X3) {$x_3$};
\node[below] at (Y3) {$y_3$};

\foreach \i in {1, 2}
{
\draw[fill] (Z\i) circle (2pt);
\draw (Z\i) circle (4pt);
}

\node[above=2pt] at (Z1) {$z_1$};
\node[below=2pt] at (Z2) {$z_2$};
\draw (Z1) -- (X1) -- (Z2);
\draw (Z1) -- (X2) -- (X3) -- (Z2);
\draw (Z1) -- (Y1) -- (Z2);
\draw (Z1) -- (Y2) -- (Y3) -- (Z2);

\path (0,2.5) coordinate (CU);
\path (4,2.5) coordinate (CV);
\draw (U) .. controls (CU) and (CV) .. (V);

\end{tikzpicture} \\
(a) & (b)
\end{tabular}

\caption{The two possible starting structures in the proof of Lemma
  \ref{lemma:deg-leq-3}. Circled vertices cannot have additional edges.}

\label{fig:deg4c}

\end{center}
\end{figure}
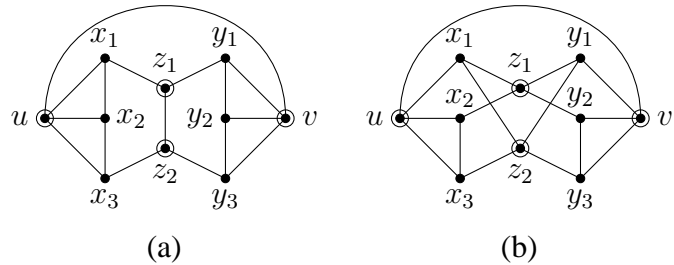

\subsection{Graphs with $\Delta(G) = n-7$}

In this section, we complete the proof of Theorem \ref{thm:petersen}.

\begin{lemma}\label{lemma:n-7}
Let $G$ be a graph with a vertex $u$ with $\Delta(G) = \deg(u) = n-7$ and such that
$\deg(v) \leq 3$ for every $v \in V \setminus N[u]$. Then either $c(G)=2$ or the induced subgraph $G[V\setminus N[u]]$ is a 6-cycle.
\end{lemma}

This lemma can be generalized a bit more. In particular, if we remove
the restriction on the maximum of degree of vertices in $V\setminus N[u]$,
then the proofs of Lemmas~\ref{lemma:gen-5-cycle} and~\ref{lemma:n-7}
can be adapted to show that $H$ must contain an induced 5-cycle or
6-cycle. However, the case analysis is cumbersome, so we have opted
for this simpler formulation. The version stated above is sufficient
to prove one our the main results: the Petersen graph is the only
10-vertex graph requiring 3 cops.

\begin{proof}[Proof of Lemma~\ref{lemma:n-7}]
Let $H=G[V\setminus N[u]]$ and suppose that $c(G) >2$. First, we observe that $H$ must be connected. Otherwise, we can adapt the proof of Corollary~\ref{cor:deg4deg3} to show that $c(G)=2$. Indeed, $H$ has at most one component $H_1$ whose cop number is 2. We use the strategy described in the proof of Corollary~\ref{cor:deg4deg3} to capture the robber. The only alteration of the strategy is to address the robber moving from $N(u)$ to $H - H_1$. However, $|V(H-H_1)| \leq 2$, so this component is cop-win. One cop responds by moving to $u$, while the other moves into $H-H_1$ for the win (by Lemma~\ref{obs:key}(a)).

Therefore, we may assume that $H$ is connected and  $c(H) \geq 2$. This means that $H$ must contain an induced $k$-cycle for $k \in {4,5,6}.$
Suppose that $G$ contains an induced 4-cycle $x_1, x_2, x_3, x_4$. Without loss of generality, $x_5 \sim x_1$, and $x_6$ is adjacent to at most three of $\{ x_2,x_3,x_4, x_5 \}$ (because we already have $\deg(x_1) =3$). Start the cops at $u$ and $x_1$, so that $S(R)$ is one of $\{ x_3 \}$, $\{ x_6 \}$ or $\{x_3, x_6\}$. In the first two cases, $\Delta(S(R)) \leq 3$ so the cops win by Corollary \ref{cor:one-deg-3}. The last option occurs when $x_3 \sim x_6$. If $x_6$ has at most one neighbor in $N(u)$, then we are again done by Corollary~\ref{cor:one-deg-3}, since $\Delta(S(R)) \leq 3$. When $x_6$ has two neighbors in $N(u)$, the game play depends on the initial location of the robber. If the robber starts at $x_6$, then $C_1$ holds at $u$ while $C_2$ moves from $x_1$ to $x_2$ to $x_3$, trapping the robber. If the robber starts at $x_3$, then the roles are reversed: $C_1$ moves to $x_6$ in two steps while $C_2$ holds at $x_1$. At this point, the robber is trapped.

Next, suppose that $G$ contains an induced 5-cycle $x_1, x_2, x_3, x_4, x_5$. Without loss of generality, $x_6 \sim x_1$. If $x_6$ is adjacent to two of the $x_i$, then we can place $C_1$ at $u$ and $C_2$ at some $x_j$ so that $|N(S(R))\cap N(u)| \leq 1$, giving a cop winning position by Lemma~\ref{obs:key}(b). Indeed, by symmetry there are only 2 cases to consider: if $x_6\sim x_2$, then $C_2$ starts at $x_4$ and $S(R) = \{x_1, x_2, x_6\}$; if $x_6\sim x_3$, then $C_2$ starts $x_3$, and $S(R) =\{x_1, x_5\}$.
So we may assume that $x_6$ has no additional neighbors in $H$. There are two cases to consider.
If $x_2$ and $x_4$ do not share a neighbor in $N(u)$, then the game play begins with $C_2$ chasing $R$ onto $x_2$:
$\copst{u,x_1}{\{x_3, x_4\}} \crmove \cdots  \crmove  \copst{u, \{x_4, x_5\}}{x_2}$. If $x_2$ is not adjacent to $N(u)$, then the cops can ensure $S(R)$ satisfies Corollary \ref{cor:one-deg-3} on their next move. Indeed, $C_2$ moves to $x_4$. If $N(x_6)\cap N(u)=\emptyset$, then the situation already satisfies Corollary~\ref{cor:one-deg-3}, otherwise, $C_1$ moves to $N(x_6)\cap N(u)$, and now the situation satisfies Corollary~\ref{cor:one-deg-3}.

The final case to consider is when $x_2$ and $x_4$ are both adjacent to $y \in N(u)$. By symmetry, $x_3$ and $x_5$ are adjacent to $z \in N(u)$. By symmetry, there is one game to consider:
$\copst{u,x_1}{x_3} \crmove \copst{z,x_2}{x_4}$ which is cop-win by Corollary \ref{cor:one-deg-3}.
Thus, the only option for $H$ is an induced 6-cycle.
\end{proof}

We can now prove that the Petersen graph is the unique 3 cop-win graph of order 10. The following lemma
may be proved by checking the 18 possible 3-regular graphs of order 10 listed at
\cite{mckay}, but we provide a short proof for completeness.

\begin{lemma}\label{fact:petersen}
The Petersen graph is the only $3$-regular graph $G$ such that for every vertex $u\in V(G)$, $G[V(G)\setminus N[u]]$ is a $6$-cycle.
\end{lemma}
\begin{proof}
 Pick any vertex $u$ in $G$. The
  complement is a 6-cycle, where every vertex is adjacent to exactly
  one vertex in $N(u)$. Let $N(u)= \{y, z, w\}$.  Label the vertices of the $6$-cycle $x_i$, $0\le i \le 5$, where edges are between consecutive indices.  Without loss of generality, say $x_0\sim y$. Because $V \setminus N[x_0]$ is a
  6-cycle, we must have that $x_2\sim w$ and $x_4\sim z$ (by symmetry
  this is the only option). The only remaining edges to add are a
  matching between $x_1, x_3, x_0$ and $y, z, w$. To avoid a triangle
  in $V\setminus N[y]$, we cannot have $x_3\sim z$ or $x_3\sim w$; hence,
  $x_3\sim y$. Similarly, $x_1\sim z$, and $x_5\sim w$. But this gives an isomorphic copy of the
  Petersen graph.
\end{proof}

We now prove one of the main theorems.

\begin{proof}[Proof of Theorem~\ref{thm:petersen}]
Let $G$ be a graph of order 10 such that $c(G)=3$.
We have $\delta(G) \geq 2$: otherwise, the vertex of degree one $v \in V(G)$ is a dominated vertex, so $c(G)=c(G-v) \leq 2$ by Theorem
\ref{thm:9-vtx-c2}. Lemma \ref{lemma:deg-leq-3} ensures that $\Delta(G) \leq 3$. It is straightforward to see that $\Delta(G)=3$ since a connected 2-regular graph is a cycle which is 2-cop-win.

  Suppose a vertex $u \in V(G)$ has
$\deg(u)=3$. Then by Lemma~\ref{lemma:n-7}, $G[V \setminus N[u]]$ must be a
6-cycle. If every vertex in $N(u)$ has degree 3, then $G$ is 3-regular
with $c(G)=3$, and therefore, $G$ is the Petersen graph by Lemma~\ref{fact:petersen}. Otherwise,
there is a vertex $x_1 \in V \setminus N[u]$ with $\deg(v)=2$. In the rest of
the proof we give a winning strategy for the cops in this case.

Let the 6-cycle $G[V \setminus N[u]]$ be $\{ x_1, x_2, x_3, x_4, x_5, x_6 \}$ with edges between consecutive indices. Without loss of generality, $\deg(x_1)=2$ and $\deg(x_2)=3$. Let $k=\max\{i\mid \deg(x_i)=3\}$. The initial configuration is
$$\copst{u,x_4}{\{x_1,x_2,x_6\}}.$$ If $k \leq 5$, then the cops win by Corollary \ref{cor:one-deg-3}.
When $k=6$, the strategy depends on the initial robber location. Let $y \in N(u) \cap N(x_2)$. We either have
$\copst{u,x_4}{x_2} \crmove \copst{y,x_4}{x_1} \copmove \robst{y,x_5}{x_1}\copwin$, or
$\copst{u,x_4}{x_1} \crmove \copst{y,x_5}{x_1}\copwin$, or
$\copst{u,x_4}{x_6} \crmove \copst{u,x_5}{x_1}\ \copmove \robst{y,x_6}{x_1}\copwin$. The robber is trapped for every initial placement.
\end{proof}

\section{Further directions}

We conclude with some reflections on our results and some open problems. The
Petersen graph is the unique 3-regular graph of girth 5 of minimal
order, so that Theorem \ref{thm:petersen} provides a tight lower bound
for $n$ when $c(G)=3$. Recall that a $(k,g)$-\emph{cage} is a $k$-regular graph with
girth $g$ of minimal order. See \cite{cage} for a survey of cages. The Petersen graph is the unique $(3,5)$-cage, and in
general, cages exist for any pair $k \geq 2$ and $g \geq 3$.
Aigner and Fromme \cite{af} proved
that graphs with girth $5$, and degree $k$ have cop number at least
$k$; in particular, if $G$ is a $(k,5)$-cage, then $c(G) \geq k$.  Let
$n(k,g)$ denote the order of a $(k,g)$-cage.  Is it true that a
$(k,5)$-cage is $k$-cop-win?  Next, since we have $m_k \geq
n(k,5)$, it is natural to speculate whether $m_k = n(k,5)$ for $k
\geq 4$. It seems reasonable to expect that this is true at least for
small values of $k$.  It is known that $n(4,5)=19$, $n(5,5)=30$,
$n(6,5)=40$ and $n(7,5)=50$. Do any of these cages attain the
analogous $m_k$?
More generally, we can ask the same question for large $k$: is $m_k$ achieved by a $(k,5)$-cage?
It is known that $n(k,5) = \Theta(k^2)$, so an affirmative resolution would be consistent with
Theorem~\ref{iii}.

The techniques to prove Theorems~\ref{thm:9-vtx-c2} and \ref{thm:petersen} may prove useful in classifying the cop number of graphs with order $11$. We will consider this problem, and the value of $m_4$ in future work.

\subsection*{Acknowledgments}
 We  thank Volkan Isler, Graeme Kemkes, Richard Nowakowski,  Pawel Pra{\l }at, and Vishal Saraswat for helpful discussions. Part of this work was supported by  Institute for Mathematics and its Applications during its Summer 2010 Special Program on Interdisciplinary Research for Undergraduates.

\end{document}